\documentclass{amsart}
\usepackage[utf8]{inputenc}

\title{Homogeneous surfaces admitting invariant connections}

\author[D. Bl\'azquez-Sanz]{David Bl\'azquez-Sanz}
\author[L. F. Jim\'enez Buitrago]{Luis Fernando Jim\'enez Buitrago}
\author[C. Mar\'in]{Carlos Alberto Mar\'in Arango}

\address{Instituto de Matem\'aticas \hfill\break\indent  Universidad de Antioquia \hfill\break\indent Medell\'in, Colombia}
\email{calberto.marin@matematicas.udea.edu.co}
\address{Universidad Nacional de Colombia - Sede Medell\'in \hfill\break\indent  Facultad de Ciencias 
\hfill\break\indent Escuela de Matem\'aticas \hfill\break\indent  Medell\'in, Colombia}
\email{dblazquezs@unal.edu.co}
\email{lufjimenezbu@unal.edu.co}



\date{June 2019}

\usepackage[english]{babel}
\usepackage{times}
\usepackage{dsfont}

\usepackage{amsmath,amssymb,hyperref,color,amsthm,verbatim}   
\usepackage[mathscr]{euscript}
\usepackage{pb-diagram}
\usepackage[all]{xy}
\usepackage{cleveref}
\usepackage{multirow}%para las tablas, lo puse yo Luis.
\usepackage{multicol}
\begin{document}

\makeatletter
\newenvironment{dem}{Proof:}{\qed}

\makeatother

\numberwithin{equation}{section}
\theoremstyle{plain}\newtheorem{theorem}{Theorem}[section]
\theoremstyle{plain}\newtheorem{proposition}[theorem]{Proposition}
\theoremstyle{plain}\newtheorem{lemma}[theorem]{Lemma}

\theoremstyle{definition}\newtheorem{definition}[theorem]{Definition}
\theoremstyle{remark}\newtheorem{remark}[theorem]{Remark}

\theoremstyle{definition}\newtheorem{example}[theorem]{Example}

\newcommand{\fn}[5]{
            \begin{array}{cccl}
                #1: & #2 & \longrightarrow &  {#3}\\
                    & #4 & \longmapsto & {#1}(#4)={#5}
            \end{array}}
            
\newcommand{\na}[3]{\left( \mathcal{L}_{#1}\nabla \right)_{#2}{#3}=0} 

\newcommand{\dx}{\partial_x}
\newcommand{\dy}{\partial_y}
\newcommand{\dxi}{\partial_{x_i}}
\newcommand{\dxj}{\partial_{x_j}}
\newcommand{\dxk}{\partial_{x_k}}
\newcommand{\dxu}{\partial_{x_1}}
\newcommand{\dxd}{\partial_{x_2}}
\newcommand{\dt}{\partial_{\theta}}
\newcommand{\dr}{\partial_{r}}

\def\homo{\mathrm{Hom}}
\def\L{\mathcal{L}}
\def\r{\mathbb{R}}
\def\n{\mathbb{N}}
\def\q{\mathbb{Q}}
\def\z{\mathbb{Z}}
\def\so{\text{SO}}
\def\o{\text{O}}
\def\s{\mathbb{S}}
\def\c{\mathbb{C}}
\def\g{\mathfrak{g}} 
\def\cp{\mathbb{CP}} 
\def\fibP{\pi : P\to M}
\def\fibV{\pi : E\to M}
\def\d{\mathrm{d}}
\def\h{\mathcal{H}}
\def\gl{{\rm gl}(n)}
\def\X{\mathfrak{X}}
\def\lin{\rm Lin}
%def\dx{\partial x}
%def\dy{\partial y}
\def\gij{\Gamma_{ij}^l}
\def\gu{\Gamma_{11}}
\def\gd{\Gamma_{22}}
\def\gud{\Gamma_{12}}
\def\gdu{\Gamma_{21}}

\maketitle

\begin{abstract}
    We compute all the simply connected homogeneous and infinitesimally homogeneous surfaces admitting one or more invariant affine connections. We find exactly six non equivalent simply connected homogeneous surfaces admitting more than one invariant connections and four classes of simply connected homogeneous surfaces admitting exactly one invariant connection.
\end{abstract}

\section{Introduction}

From the XIXth century on geometry has been understood by means of its transformation groups. An \emph{homogeneous} space $M$ is a smooth manifold endowed with a \emph{transitive smooth action} of a Lie group $G$. In particular, if $H$ is the stabilizer subgroup of a point $x\in M$ then the space $M$ can be recovered as the coset space $G/H$.  Since the group $G$ acting in $M$ is seen as a transformation group, i.e., a subgroup of ${\rm Diff}(M)$ there is no loss of generality in assuming, whenever it is required, that the action of $G$ in $M$ is faithful.

%\medskip
The notion of homogeneous space has an infinitesimal version. A \emph{infinitesimal homogeneous space} is a manifold endowed with a \emph{Lie algebra representation} of a Lie algebra $\mathfrak g$ into the Lie algebra of smooth vector fields $\mathfrak X(M)$. As before, since the lie Lie algebra $\mathfrak g$ is seen as a Lie algebra of vector fields, there is no loss of generality in assuming, whenever it is required, that the representation is faithful.

%\medskip
The heuristics of the celebrated F. Klein's \emph{Erlangen program} is that geometry should studied through the invariants of the action of $G$ on $M$. As example, the invariants of the group of euclidean movements in the euclidean space are the distances, volumes, angles, etc.

%\medskip
An \emph{affine connection} in a manifold $M$ is certain geometric construction that allows a notion of parallel transport of vectors. Affine connections are a kind of geometric structures, and then they are transformed covariantly by diffeomorphisms. The problem of construction of invariant connections in an homogeneous space has been treated by S. Kobayashi and K. Nomizu in a series of papers that are summarized in \cite{kobayashi1996foundations2} chapter X. In several examples of homogeneous spaces, there is a unique invariant connection. 

%\medskip 
We ask the following question.  Which homogeneous spaces of a fixed dimension do admit exactly one, or strictly more than one invariant connections? It is possible to give a complete list? We can solve this problem for low dimension manifolds. We start with the local problem, for infinitesimal homogeneous spaces. The point is that the structures of infinitesimal homogeneous spaces admitted by a surface are completely classified. S. Lie, who completed in \cite{Lie} the task for Lie algebras of analytic vector fields in $\mathbb C^2$ and $\mathbb C^3$. The local classification of finite dimensional Lie algebras of smooth vector fields in $\mathbb R^2$ was completed by González-López, Kamran and Olver in \cite{Olver}. The real list is larger than the complex one as some complex Lie algebra may have several possible real forms. A completely different approach leading to the same classifications was followed by Komrakov, Churyumov and Doubrov in \cite{komrakov}. 
%\medskip
By means of the Lie derivative of connections (Definition \ref{df:LD_connection} and Lemma \ref{lm:Lie_connection_zero}) we compute the space of invariant connections by an infinitesimal action of Lie algebra. Then, 
we go through the list of all possible structures of infinitesimal homogeneous surfaces (Tables \ref{tabla:acciones primitivas}, \ref{tabla:acciones no transitivas}, \ref{tabla:imprimitive actions}) obtaining the spaces of invariant connection for each local class of infinitesimal actions (Theorems \ref{th:primitive} and \ref{th:imprimitive}). We also perform the computation of the space of invariant connections for non-transitive actions, for the shake of completeness (Propositions \ref{pr:no_transitive1} and \ref{pr:no_transitive2}). Finally we integrate our results to obtain homogeneous surfaces admitting invariant connections. We obtain that there is a finite list of simply connected homogeneous surfaces admitting invariant connections. In particular there are $6$ non-equivalent simply connected homogeneous surfaces admitting more than one invariant connection. Our main results are the following:

\medskip

\noindent{\bf Theorem \ref{th:hs}} \emph{
Let $M$, endowed with an action of a connected Lie group $G$, be a simply connected homogeneous surface. Let us assume that $M$ admits at least two $G$-invariant connections. Then, $M$ is equivalent to one of the following cases:
\begin{enumerate}
\item[(a)] The affine plane $\mathbb R^2$ with one of the following transitive groups of affine transformations:
$${\rm Res}_{(2:1)}(\mathbb R^2) = \left\{ \left[ \begin{array}{c}x\\ y\end{array}\right] \mapsto A \left[ \begin{array}{c}x\\ y\end{array}\right] +  \left[ \begin{array}{c}b_1\\ b_2\end{array}\right]\mid A = 
\left[ \begin{array}{cc} \lambda^2 & 0 \\ 0 & \lambda \end{array} \right],\,\lambda>0
\right\},$$
$${\rm Trans}(\mathbb R^2) = \left\{  \left[ \begin{array}{c}x\\ y\end{array}\right] \mapsto  \left[ \begin{array}{c}x\\ y\end{array}\right] +  \left[ \begin{array}{c}b_1\\ b_2\end{array}\right] \mid b_1,b_2\in\mathbb R \right\},$$
$${\rm Res}_{(0:1)}(\mathbb R^2) = \left\{  \left[ \begin{array}{c}x\\ y\end{array}\right] \mapsto A  \left[ \begin{array}{c}x\\ y\end{array}\right] +  \left[ \begin{array}{c}b_1\\ b_2\end{array}\right]\mid A = 
\left[ \begin{array}{cc} 1 & 0 \\ 0 & \lambda \end{array} \right],\,\lambda>0
\right\}.$$
%$G/H$ where $G = \mathbb R\ltimes \mathbb R^2$ with the semi-direct product %law 
%$$(x,y,z)\star(x',y',z') = (x+x',y + e^xy', z + e^{x/2}z'),$$ 
%and $H =  \mathbb R  \ltimes 0.$ 
\item[(b)] ${\rm SL}_2(\mathbb R)/U$ where $U$ is the subgroup of superior unipotent matrices
$$ U = \left\{ A\in {\rm SL}_2(\mathbb R) \mid A = 
\left[ \begin{array}{cc} 1 & \lambda \\ 0 & 1 \end{array} \right]
\right\}.$$
\item[(c)] $\mathbb R\ltimes \mathbb R$ acting on itself by left translations. 
\item[(d)] $G/H$ with $G = \mathbb R^2\ltimes \mathbb R$ and $H = (\mathbb R\times 0)\ltimes 0$. 
\end{enumerate}
%In such cases we also have that the dimension of the space of $G$-invariant connections is $1$, $3$, $8$ and $4$ respectively. 
}

\medskip

\noindent{\bf Theorem \ref{th:hs2}} \emph{ 
Let $M$, with the action of $G$, be a simply connected homogeneous surface. Let us assume that $M$ admits exactly one $G$-invariant connection. Then, $M$ is equivalent to one of the following cases:
\begin{enumerate}
\item[(a)] The affine plane $\mathbb R^2$ with $G$ any connected transitive subgroup of the group ${\rm Aff}(\mathbb R^2)$ of affine transformations containing the group of translations and not conjugated with any of the groups,
$${\rm Res}_{(2:1)}(\mathbb R^2) = \left\{ \left[ \begin{array}{c}x\\ y\end{array}\right] \mapsto A \left[ \begin{array}{c}x\\ y\end{array}\right] +  \left[ \begin{array}{c}b_1\\ b_2\end{array}\right]\mid A = 
\left[ \begin{array}{cc} \lambda^2 & 0 \\ 0 & \lambda \end{array} \right],\,\lambda>0
\right\},$$
$${\rm Trans}(\mathbb R^2) = \left\{  \left[ \begin{array}{c}x\\ y\end{array}\right] \mapsto  \left[ \begin{array}{c}x\\ y\end{array}\right] +  \left[ \begin{array}{c}b_1\\ b_2\end{array}\right] \mid b_1,b_2\in\mathbb R \right\},$$
$${\rm Res}_{(0:1)}(\mathbb R^2) = \left\{  \left[ \begin{array}{c}x\\ y\end{array}\right] \mapsto A  \left[ \begin{array}{c}x\\ y\end{array}\right] +  \left[ \begin{array}{c}b_1\\ b_2\end{array}\right]\mid A = 
\left[ \begin{array}{cc} 1 & 0 \\ 0 & \lambda \end{array} \right],\,\lambda>0
\right\}.$$
(this case includes the euclidean plane)
\item[(b)] ${\rm SL}_2(\mathbb R)/H$ where $H$ is the group of special diagonal matrices,
$$H = \left\{ A \in {\rm SL}_{2}(\mathbb R) \mid  
A =  \left[ \begin{array}{cc} \lambda & 0 \\ 0 & \lambda^{-1} \end{array} \right], \,\, \lambda>0 \right\}.$$
\item[(c)] The hyperbolic plane 
$$\mathbb H = \{z\in\mathbb C \mid {\rm Im}(z)>0 \}$$ 
with the group ${\rm SL}_2(\mathbb R)$ of hyperbolic rotations.
\item[(d)] Spherical surface, 
$$S^2 = \{(x,y,z)\in \mathbb R^3\mid x^2 + y^2 + z^2 =1 \}$$ 
with its group ${\rm SO}_3(\mathbb R)$ of rotations.
\end{enumerate}
}

\medskip

Some related work has been presented by Kowalski, Opozda, Vl\'a\v{s}ek and Arias-Marco \cite{kowalski2008, kowalski2004, opozda2004} who gave answer to a dual question in the local case: what are the local canonical forms of affine connections in surfaces whose symmetries act transitively? Their methods are similar to ours in the following sense: their results and ours are obtained through a careful analysis of the local classification of Lie algebra actions given in \cite{Olver}.

\section{Invariant connections}

Here we review the construction of the space of connections, the bundle of connections, and the action of diffeomorphisms and vector fields on them, see for instance \cite{castrillon, garcia}. Let $M$ be a smooth manifold, $\mathfrak X(M)$ the Lie algebra of smooth vector fields in $M$ and ${\rm End}_M(\mathfrak X(M))$ the module of $\mathcal C^{\infty}(M)$-linear endomorphisms of $\mathfrak X(M)$. An \emph{affine connection} in $M$ is a $\mathbb R$-linear map
$$\nabla \colon \mathfrak X(M)\to {\rm End}_M(\mathfrak X(M)),$$
such that $\nabla(fX) = df\otimes X + f\nabla X$.\footnote{In order to keep the usual notation for covariant derivatives we have $\nabla(X)(Y) = \nabla_Y X$.}
Let ${\rm Cnx}(M)$ the space of all affine connections in $M$. It is easy to check that if $\nabla\in {\rm Cnx}(M)$ is an affine connection and $\theta\in \Omega^1(M,{\rm End}(TM))$ is a $1$-form in $M$ with values endomorphisms of $TM$ then $\nabla+\theta$ is an affine connection in $M$. Indeed, the difference between two connections in a $1$-form with values endomorphisms of $TM$. Therefore, we have a free and transitive additive action,
$$\Omega^1(M,{\rm End}(TM))\times {\rm Cnx}(M) \to {\rm Cnx}(M),\quad
(\theta,\nabla)\to \nabla + \theta,$$
that gives to ${\rm Cnx}(M)$ the structure of an affine space modeled over the vector space $\Omega^1(M,{\rm End}(TM))$.%\footnote{As $\Omega^1(M,{\rm End}(TM))$ is the space of sections of the vector bundle $T^*M\otimes {\rm End}(TM)$ it is possible to see the affine connections as sections of an affine bundle modeled over the aforementioned vector bundle.}

It is clear the connections are local objects, they can be restricted to open subsets of $M$, and they are determined by its restrictions to a covering family of open subsets. The following construction is local.  
Given a global frame\footnote{That is, and ordered base of sections of the tangent bundle of $M$.} $s = (X_1,\ldots,X_n)$ in $M$ there is an associated connection $\nabla^s$ such that $\nabla^s(X_i)=0$ for $i\in \{1,\ldots,n\}$.
We taking $\nabla^s$ as the initial point of ${\rm Cnx}(M)$ we have an identification,
$$\Gamma_s \colon {\rm Cnx}(M) \xrightarrow{\sim} \Omega^1(M,{\rm End}(TM)), \quad \nabla \mapsto \Gamma_s(\nabla) = \nabla - \nabla^s.$$
Here $\Gamma_s(\nabla)$ is the so-called Christoffel tensor of $\nabla$ with respect to the frame $s$. We say that two connections $\nabla$ and $\bar\nabla$ take the same value at the point $p$ of $M$ if their Christofell tensors with respect to any frame coincide on $p$. The set $CM$ of the values at points of $M$ of affine connections is therefore an affine bundle over $M$ modeled over $T^*M\otimes {\rm End}(TM)$. An affine connection $\nabla$ is thus seen in a canonical way as a global section of the affine bundle $CM\to M$.

Given a diffeomorphism $f\colon M\to N$ the push forward of connections is an isomorphism $f_*\colon {\rm Cnx}(M)\to {\rm Cnx}(N)$ defined in terms of the push forward of vector fields by $(f_*\nabla)_XY = f_*(\nabla_{f^{-1}_*X} f^{-1}_*Y)$. The push forward of connections is compatible with the affine structure, in the sense that $f_*(\nabla + \theta) = f_*(\nabla) + f_*(\theta)$ for a connection $\nabla$ and a $1$-form of endomorphisms $\theta$.

If we look at the connections as sections of the affine bundle $CM$, then we have that the push forward is given by a natural transformation $Cf$ of affine bundles,
\begin{equation}\label{eq:diagrama1}
\xymatrix{ CM\ar[d] \ar[r]^-{Cf} & CN \ar[d] \\ 
M \ar[r]^-{f} \ar@/^20pt/[u]^-{\nabla} & N \ar@/^-20pt/[u]_-{f_*\nabla}} 
\end{equation}
here $(Cf)(\nabla(p)) = (f_*\nabla)(f(p))$.

Let us fix an smooth action\footnote{By convention, we consider the action on the left side, we use the standard notation $gp = L_g(p)$.} of a Lie group $G$ on $M$. A tensor $\theta$ is $G$-invariant if $(L_g)_*\theta = \theta$ for all $g$ in $G$. The space of $G$-invariant $1$-forms of endomorphisms is denoted by $\Omega^1(M,{\rm End}(TM))^G$. Analogously, we say that a connection $\nabla\in{\rm Cnx}(M)$ is $G$-invariant if $(L_g)_*\nabla=\nabla$, for all $g\in G$; we denote by ${\rm Cnx}(M)^G$ the set of $G$-invariant connections in $M$. It is clear that the sum of a $G$-invariant connection and a $G$-invariant $1$-form of endomorphisms yields a $G$-invariant connection. So that, ${\rm Cnx}(TM)^G$ is either empty or an a affine space modeled over the space  $\Omega^1(M,\textrm{End}(TM))^G$. 

Let us denote by $G_p$ the stabilizer\footnote{That is, the subgroup of $G$ formed by the elements $g\in G$ such that $gp = p$.} of a point $p\in M$. We have an action of $G_p$ on the fiber $(CM)_p$ by affine transformations.
There is also a natural linear representation of $G_p$ in $T_p^*M \otimes {\rm End}(T_pM)$.

\begin{proposition}\label{pr:inv1}
 If $G$ acts transitively on $M$ then ${\rm Cnx}(M)^G$ has finite real dimension $\leq (\dim M)^3$. In such case:
 \begin{enumerate}
     \item ${\rm Cnx}(M)^G \simeq (CM)_p^{G_p}$ for any $p\in M$
     \item If ${\rm Cnx}(M)^G\neq \emptyset$ then
     $\dim {\rm Cnx}(M)^G = \dim \left(T_p^*M \otimes {\rm End}(T_pM)\right)^{G_p}$.
 \end{enumerate}
\end{proposition}

\begin{proof}
(1) There is a natural map ${\rm Cnx}(M)\to (CM)_p$ of evaluation at $p$, $\nabla\mapsto \nabla(p)$. It is clear that it maps ${\rm Cnx}(M)^G$ on $(CM)_p^G$. By the action of $G$ on $M$ we construct an inverse map,
$$(CM)_p^G \xrightarrow{\,\sim\,} {\rm Cnx}(M), \quad c_p\mapsto \nabla$$
where $\nabla$ is defined as $\nabla(gp) = C(L_g)(c_p)$.

(2) It is also clear that if $(CM)_p$ is a linear affine space modeled over $T_p^*M \otimes {\rm End}(T_pM)$. The action of $G_p$ is compatible with the affine structure, therefore $(CM)_p^G$, if not empty, is a linear affine space modeled over $(T_p^*M \otimes {\rm End}(T_pM))^{G_p}$. 
\end{proof}

%Prop:
%A toda acci\'on le corresponde un espacio de conexiones invariantes. Si la acci\'on es transitiva, el espacio es de dimensi\'on finita.
The previous situation can be considered infinitesimally. Let us recall that an action of a Lie algebra $\g$ in $M$ is a Lie algebra morphism  $\phi:\g \to \mathfrak{X}(M)$. Without loss of generality we may assume that the action is faithful, that is, that $\phi$ is inyective. In such case we see $\mathfrak g\subset\mathfrak X(M)$ as a Lie algebra of vector fields in $M$. Therefore, from now on we will consider the elements of $\mathfrak g$ as vector fields in $M$.

We say that a tensor $\theta$ in $M$ is $\mathfrak g$-invariant if for any local diffeomorphism $\sigma \colon U\xrightarrow{\sim} V$ that belongs to the flow pseudogroup of a vector field $X\in \mathfrak g$ we have $\sigma_*(\theta)|_{V} = \theta|_{V}$. The analogous definition of $\mathfrak g$-invariance applies to connections. As before, the space of $\mathfrak g$-invariant connections ${\rm Cnx}(M)^{\mathfrak g}$ is either empty of an affine space modeled over $\Omega^1(M,\textrm{End}(T(M)))^{\mathfrak g}$. 

Let us recall that the $\mathfrak g$ acts transitively on $M$ if for all $p\in M$ the map $\phi_p\colon \mathfrak g\to T_pM$ $X\mapsto X_p$ is surjective. This implies that, for connected $M$, the flow pseudogroup of the action is also transitive.

Diagram \eqref{eq:diagrama1} in the particular case of $f = \sigma_t$, the time $t$ flow of a vector field $X$ in $M$, gives us the prolongation $\tilde X\in \mathfrak X(CM)$,
$$\tilde X = \left.\frac{d}{dt}\right|_{t=0}C\sigma_t.$$
It is clear that $\tilde X$ projects onto $X$ and the flow of $\tilde X$ gives affine transformations between the fibers of $CM\to M$. For a given $p\in M$ the kernel $\ker(\phi_p) = \mathfrak g_p$ is the so-called \emph{stabilizer algebra} of the point $p$. There is a natural representation,
$$\tilde \phi|_{CM_p}\colon \mathfrak g_p \to \mathfrak X(CM_p)$$
the set of zeroes in $CM_p$ of these vector fields is an affine subspace denoted $CM_p^{\mathfrak g_p}$.

\begin{proposition}\label{prop:finitedim_inf}
Let us assume that $\mathfrak g$ acts transitively in a connected manifold $M$. Then the space ${\rm Cnx}(M)^{\g}$ has finite real dimension. In such case
${\rm Cnx}(M)^{\g} \simeq  CM_p^{\mathfrak g_p}$ for any $p\in M$.
\end{proposition}

\begin{proof}
It is a consequence of the same argument we exposed in Proposition \ref{pr:inv1}, but in this case we use the flow pseudogroup of $\mathfrak g$ instead of the Lie group $G$.
\end{proof}

Let us recall that given an action of $G$ in $M$ there is an associated infinitesimal action of ${\rm Lie}(G)$ given by,
$$\phi(A)_p = \left.\frac{d}{dt}\right|_{t=0} \exp(tA)\star p.$$
There is a natural relation between the invariants of an action and that of its associated infinitesimal action. In particular, for invariant connections, we have that if $G$ is a connected Lie group acting on $M$ then ${\rm Cnx}(M)^G$ $=$ ${\rm Cnx}(M)^{{\rm Lie}(G)}$.

%A toda acci\'on infinitesima le corresponde un espacio de conexiones invariantes. Si la acci\'on es transitiva, el espacio es de dimensi\'on finita. 

\subsection{Lie derivative of a connection}

In this section we develop a method for computing the invariant connections of an infinitesimal action. The well known notion of Lie derivative for tensors can be extended to sections of natural bundles as in \cite{Kolar}. In the particular case of connections, it can be defined by as follows. 

\begin{definition}\label{df:LD_connection}
The Lie derivative of a connection $\nabla$ in the direction of a field $X$ is defined as:
$$
\L_X\nabla = \lim_{t\to 0} \frac{\sigma_{-t*}\nabla - \nabla}{t}, 
$$
where $\{\sigma_t\}_{t\in\mathbb R}$  is the flow of the field $X.$
\end{definition}

Note that for an affine connection $\nabla$ and a vector field $X$ the Lie derivative $\mathcal L_X\nabla\in \Omega^1(M,{\rm End}(TM))$ is not a connection but a $1$-form of endomorphisms.

\begin{proposition}
The Lie derivative of $\nabla$ has the following properties: 
\begin{enumerate}\label{leibniz}
    \item (Linearity) $\L_{fX+gY}\nabla = f(\L_X\nabla) + g(\L_Y\nabla)$.
    \item (Leibniz formula) $(\L_X\nabla)(Y,Z) = [X,\nabla_Y Z] - \nabla_{[X,Y]}Z - \nabla_Y[X,Z]$.
\end{enumerate}
\end{proposition}
\begin{proof}
(1) can be easily checked from the definition. Let us prove (2). Without lose of generality let us assume that $X$ is complete and $\{\sigma_t\}_{t\in\mathbb R}$ is its flow. 
$$(\L_X\nabla)(Y,Z) = \lim_{t\to 0} \frac{\sigma_{-t*}\nabla - \nabla}{t}(Y,Z)
= \lim_{t\to 0} \frac{(\sigma_{-t*}\nabla)_YZ - \nabla_YZ}{t} = $$
$$ \lim_{t\to 0} \frac{\sigma_{-t*}(\nabla_YZ) - \nabla_YZ}{t} + \lim_{t\to 0} \frac{\sigma_{-t*}(\nabla_{\sigma_{t*}Y}\sigma_{t*}Z) - \sigma_{-t*}(\nabla_YZ)}{t} =$$
$$[X,\nabla_YZ] +  \lim_{t\to 0} \sigma_{-t*} \frac{\nabla_{\sigma_{t*}Y}\sigma_{t*}Z - \nabla_YZ}{t} = $$
$$ [X,\nabla_Y Z]+ \lim_{t\to 0} \frac{\nabla_{\sigma_{t*}Y}\sigma_{t*}Z - \nabla_Y\sigma_{t*}Z}{t} + 
\lim_{t\to 0}\frac{\nabla_Y\sigma_{t*}Z - \nabla_YZ}{t} = $$
$$[X,\nabla_Y Z] + \lim_{t\to 0} \nabla_{\frac{\sigma_{*t}Y-Y}{t}} \sigma_{t*}Z + \lim_{t\to 0} \nabla_Y \frac{\sigma_{t*}Z-Z}{t}=$$
$$[X,\nabla_Y Z] - \nabla_{[X,Y]}Z - \nabla_Y[X,Z].$$
\end{proof}

\begin{lemma}\label{lm:Lie_connection_zero}
Let $X$ be a vector field in $M$
and $\nabla$ be a connection. The following are equivalent:
\begin{enumerate}
    \item $\L_X \nabla = 0$.
    \item $\nabla$ is invariant by the flow of $X$.
\end{enumerate}

\begin{proof}
$(2)\Rightarrow(1)$ is clear from the definition. Let us see $(1)\Rightarrow(2)$.
Without lose of generality let us assume that $X$ is complete, and let $\sigma_t$ be its time $t$ flow. Let us define $\Gamma(t) = \sigma_{t*}(\nabla)-\nabla$. By hypothesis we have $\Gamma(0) = 0$ and $\left.\frac{d}{dt}\right|_{t=0}\Gamma(t)=0$. Moreover, for any $t$, 
$$\frac{d}{dt}\Gamma(t) = 
\lim_{s\to 0} \frac{\Gamma(t+s) - \Gamma(t)}{s} =
\lim_{s\to 0}
\frac{\sigma_{(t+s)*}\nabla - \sigma_{t*}\nabla}{s} = \sigma_{t*}(- \L_X(\nabla)) = 0.$$
therefore, $\frac{d}{dt}\Gamma(t) = 0$ and thus $\Gamma(t) = 0$.
\end{proof}

\end{lemma}\label{lm:inv}
\begin{theorem}
Let us consider $\mathfrak g$  a Lie algebra of smooth vector fields in a manifold $M$. The following are equivalent:
\begin{enumerate}
    \item $\nabla$ is $\mathfrak g$-invariant.
    \item $\L_{A}\nabla = 0$, for all $A\in\mathfrak g$.
    \item  $\L_{A_i}\nabla = 0$, for all $i=1,\ldots,n$, where $\{A_1,\ldots,A_n\}$ is a system of generators of  $\mathfrak g$ as Lie algebra.
\end{enumerate}
\end{theorem}

\begin{proof}
If $\nabla$ is $\mathfrak g$-invariant, then $\sigma_{-t*}(\nabla) = \nabla$ for the flow of any vector field of the form $\phi(A)$ with $A\in\mathfrak g$. By definition of Lie derivative we have $\L_{A}\nabla = 0$, and $(1)\Rightarrow(2)$. We also have $(2)\Rightarrow (3)$. Finally to prove $(3)\Rightarrow (1)$ it is enough to note that, from Lemma \ref{lm:inv}, $\nabla$ is invariant with respect to the flow of the vector fields $A_i$. Thus, it is invariant with respect to any composition of these flows, and then with respect to the flow of any vector field $A\in\mathfrak g$.
\end{proof}

\begin{lemma}
Let us consider a Lie algebra $\mathfrak g$ of vector fields in an open subset
$U\subseteq \mathbb R^n$. If $\mathfrak g$ contains the translation vector field $\partial_{x_i}$ then for any $\mathfrak g$-invariant connection $\nabla$ the Christoffel symbols $\Gamma_{ij}^k$ are independt from $x_i$.
\end{lemma}
\begin{proof}
Using the Leibniz formula of \ref{leibniz} we have 
$$
\left(\L_{\dxi} \nabla \right)_{\dxj} {\dxk} = \left[\dxi,\nabla_{\dxj}\dxk\right] - \nabla_{\left[\dxi,\dxj\right]}\dxk - \nabla_{\dxj}\left[\dxi,\dxk\right],
$$
\noindent for $i,j, k \in \{1,2\}$. As $\left[\dxi,\dxj\right]=0$, for any $i, j \in \{1,2\}$, the previous equality is reduced to 
$$
\left(\L_{\dxi} \nabla \right)_{\dxj} {\dxk} = \left[\dxi,\nabla_{\dxj}\dxk\right]. 
$$
If $\left[\dxi,\nabla_{\dxj}\dxk\right]=0$ then 
$\dxi \Gamma_{jk}^l = 0$, for $i,j,k,l \in \{1,2\}$.
\end{proof}

\begin{remark}\label{constante}
If the  Lie algebra contains all the translations  $\dxi$, then the Christoffel symbols are constants.
Moreoever if all of them are null, the only invariant connection is the usual one.
\end{remark}

\subsection{Classification of finite dimensional Lie algebra actions on germs of surfaces}

%\renewcommand{\tablename}{TABLA}
%The infinitesimal homogeneous spaces
The local classification of infinitesimal actions of finite dimensional Lie algebras in manifolds of complex dimension $2$ and $3$ was given by S. Lie in the XIX century. The real classification in dimension $2$ was completed in \cite{Olver} and  \cite{komrakov}. 

In Tables \ref{tabla:acciones primitivas}, \ref{tabla:acciones no transitivas} and \ref{tabla:imprimitive actions} we reproduce\footnote{ 
In Tables \ref{tabla:acciones primitivas},  \ref{tabla:acciones no transitivas} and \ref{tabla:imprimitive actions} we respect the numeration and cases of \cite{Olver}. However, since we diferenciate the transitive and non-transitive cases, some cases are listed in different order. This explains the gaps in the tables.
}
the results of \cite{Olver}: the local classification of  faithful actions of Lie algebras of finite dimension over opens of the real plane. 

\begin{table}[ht]
\centering
\begin{tabular}{p{0.1cm} p{9cm} p{2.5cm}}
%\hline
%\\
%&I. Primitive actions &\\
%\\
\hline
& Generators  & Structure   \\
\hline %\hline
1.& $\{\dx,\dy,\alpha(x\dx+y\dy)+y\dx-x\dy\}$, $\alpha \geqslant 0 $ &  $\r \ltimes \r^2$ \\
2.& $\{\dx,x\dx+y\dy,\left(x^2-y^2\right)\dx+2xy\dy\}$ & $\mathfrak{sl}(2)$ \\
3.& $\{y\dx-x\dy, (1+x^2-y^2)\dx+2xy\dy,2xy\dx+(1+y^2-x^2)\dy\}$ & $\mathfrak{so}(3)$ \\
4.& $\{\dx,\dy,x\dx+y\dy,y\dx-x\dy\}$ & $\r^2 \ltimes \r^2 $ \\
5.& $\{\dx, \dy,x\dx-y\dy, y\dx, x\dy \}$& $\mathfrak{sl}(2)\ltimes \r^2$  \\
6.& $\{\dx, \dy, x\dx, y\dy, x\dy,y\dy \}$& $\mathfrak{gl}(2)\ltimes \r^2$ \\
7.& $\{\dx, \dy, x\dx+y\dy, y\dx-x\dy, \left(x^2-y^2\right)\dx+2xy\dy, 2xy\dx+(y^2-x^2)\dy \}$& $\mathfrak{sl}(3,1)$ \\
8.& $\{\dx, \dy, x\dx, x\dy, y\dx, y\dy, x^2\dx+xy\dy, xy\dx+y^2\dy\}$& $\mathfrak{sl}(3)$ \\
\hline
\end{tabular}
\medskip 
\caption{Local classification of faithful primitive actions of finite dimensional Lie algebras on the real plane \cite{Olver}.}
\label{tabla:acciones primitivas}
\end{table}

Table \ref{tabla:acciones primitivas} contains the classification of primitive actions. Let us recall that a Lie algebra action in a surface $M$ is primitive if the induced action in the projective bundle of directions in $M$, $\mathbb{P}(T_pM)$ has no fixed points. This is equivalent to say that the action is not by infinitesimal symmetries of a foliation in $M$. In physical terms, it also means that the we consider an \emph{isotropic} geometry in the surface.  Therefore, the cases 1-8 correspond to classical 2-dimensional geometries, namely:
\begin{enumerate}
\item It depends on the parameter $\alpha$. The case $\alpha=0$ corresponds to the euclidean geometry. The case $\alpha\neq 0$ correspond to a primitive subgroup of the affine transformations of the complex plane spanned by translations and the exponential of a complex number of norm different from $1$.
\item Hyperbolic transformations of the half plane.
\item  Rotations of the sphere.
\item  Affine transformations of the complex plane.
\item  Volume preserving affine  transformations of the real plane.
\item  Affine transformations of the real plane.
\item  Conformal transformations of the Riemann sphere.
\item  Projective transformations of the real plane.
\end{enumerate}

\begin{table}[ht]
\centering
\begin{tabular}{p{0.1cm} p{9cm} p{2.5cm}}
%\hline
%\\
%&II. Imprimitive actions  & \\
%\\
\hline
& Generators   & Structure  \\
\hline %\hline
9. & $\{\dx\}$ & $\r$\\
10.& $\{\dx, x\dx \}$ & $\mathfrak{h}_2$\\
11. & $\{\dx, x\dx, x^2\dx\}$& $\mathfrak{sl}(2)$\\
20. & $\{\dy, \xi _1(x)\dy,\cdots, \xi _r(x)\dy\}, r\geq 1$& $\r^{r+1}$\\
21. & $\{\dy, y\dy, \xi_1(x)\dy,\cdots, \xi_r(x)\dy\}$, con $r\geq 1$ & $\r\ltimes \r^{r+1}$\\
\hline
\end{tabular}
\medskip 
\caption{Local classification of non-transitive faithful actions of finite dimensional Lie algebras in the real plane \cite{Olver}.
Functions $\xi_j$ are linearly independent. }
\label{tabla:acciones no transitivas}
\end{table}
 
\begin{table}[ht]
\centering
\begin{tabular}{p{0.1cm} p{9.3cm} p{2.5cm}}
%\hline
%\\
%&II. Imprimitive actions  & \\
%\\
\hline
& Generators   & Structure  \\
\hline %\hline
12. & $\{\dx, \dy,x\dx+\alpha y\dy\}$& $\r \ltimes \r^2$\\
13. & $\{\dx,\dy, x\dx, y\dy \}$& $\mathfrak{h}_2\oplus \mathfrak{h}_2$ \\
14. & $\{\dx, \dy, x\dx, x^2\dx\}$& $\mathfrak{gl}(2)$\\
15. & $\{\dx,\dy, x\dx, y\dy, x^2\dx, y^2\dy\}$& $\mathfrak{sl}(2)\oplus \mathfrak{sl}(2)$\\
16. & $\{\dx,\dy, x\dx, y\dy, x^2\dx,\}$&$\mathfrak{sl}(2)\oplus \mathfrak{h}_2$ \\
17. & $\{\dx+\dy,  x\dx+y\dy, x^2\dx+y^2\dy  \}$& $\mathfrak{sl}(2)$\\
18. & $\{\dx, 2x\dx+y\dy, x^2\dx+xy\dy\}$& $\mathfrak{sl}(2)$ \\
19. & $\{\dx, x\dx, y\dy, x^2\dx+xy\dy\}$& $\mathfrak{gl}(2)$\\
22. & $\{\dx, \eta_1(x)\dy, \cdots, \eta_r(x)\dy \}, r\geq 1$& $\r\ltimes \r^{r} $\\
23. & $\{\dx, y\dy, \eta_1(x)\dy, \cdots, \eta_r(x)\dy   \}, r\geq 1$& $\r^2\ltimes \r^r$\\
24. & $\{\dx, \dy, x\dx+\alpha y\dy, x\dy, \cdots, x^r\dy\}, r\geq 1$
& $\mathfrak{h}_2\ltimes \r^{r+1}$\\
25. & $\{\dx, \dy, x\dy, \cdots, x^{r-1}\dy, x\dx+\left(ry+x^r\right)\dy \}, r\geq 1$& $\r \ltimes(\r\ltimes \r^r)$\\
26. & $\{\dx, \dy, x\dx, x\dy, y\dy, x^2\dy, \cdots, x^r\dy \}, r\geq 1$ & $(\mathfrak{h}_2\oplus \r)\ltimes \r^{r+1}$\\
27. & $\{\dx, \dy, 2x\dx+ry\dy, x\dy, x^2\dx+rxy\dy, x^2\dy, \cdots, x^r\dy\}, r\geq 1$& $\mathfrak{sl}(2)\ltimes \r^{r+1}$\\
28. & $\{\dx, \dy, x\dx, x\dy, y\dy, x^2\dx+rxy\dy, x^2\dy,\cdots, x^r\dy\}, r\geq 1$& $\mathfrak{gl}(2)\ltimes \r^{r+1}$\\
\hline
\end{tabular}
\medskip 
\caption{Local classification of faithful transitive actions of finite dimensional Lie algebras on the real plane \cite{Olver}. Functions  $\eta_j$ are a base of solutions of a lineal differential equation of order $r$, with constant coefficients.
}
\label{tabla:imprimitive actions}
\end{table}

Table \ref{tabla:acciones no transitivas} contains the local classification of non transitive actions. They are, by necessity, imprimitive. Here we can find the classical one-dimensional goeometries: euclidean (case 9), affine (case 10) and projective (case 11). In such cases Proposition \ref{prop:finitedim_inf} does not apply and we will obtain infinite dimensional spaces of invariant connections. Finally, Table \ref{tabla:imprimitive actions} contains the local classification of transitive imprimitive actions.

Finally we will compute the spaces of invariant connections for lie algebra actions on connected surfaces.

\section{Invariant connections for Lie algebra actions on connected surfaces}

\subsection{Primitive Lie algebra actions}

For this classical cases, the existence of invariant connections is well known. By application of Lemma \ref{lm:inv}, we can also recover the following result.

%\textcolor{red}{Reescribir para tener en cuenta el caso 1. con $\alpha\geq 0$.}
\begin{theorem}\label{th:primitive}
Let $M$ be a connected surface endowed with faithful primitive Lie algebra action of a finite dimensional Lie algebra $\mathfrak g$ there is either one or none $\mathfrak g$-invariant connection. Moreover, only one of the following cases hold:
\begin{itemize}
\item[(a)] $\mathfrak g$ is isomorphic to a Lie sub algebra of the Lie algebra of infinitesimal affine transformations of the plane. There is one $\mathfrak g$-invariant connection; it is flat and torsion free.
\item[(b)] $\mathfrak g$ is isomorphic to the Lie algebra of infinitesimal isometries of a surface of non vanishing constant curvature. There is one $\mathfrak g$-invariant connection; it is torsion free and of constant curvature.
\item[(c)] $\mathfrak g$ is isomorphic to either the Lie algebra of infinitesimal conformal transformations of the Riemann sphere or the Lie algebra of infinitesimal projective transformations of the real plane. There are no $\mathfrak g$-invariant connections.
\end{itemize}
\end{theorem}

\begin{proof}
As in the case (1), $\dx$ and $\dy$ are in $\g$ then $\Gamma_{i,j}^k$ are constants. If we consider $X=\alpha(x\dx+y\dy)+y\dx-x\dy$. Since $\left(L_X \nabla \right)_{\dxi}\dxj=0$  we have the following equations: 
\begin{multicols}{2}

\begin{align*}
\ \alpha \gu^1-\gu^2-\gud^1-\gdu^1&=0\\ 
\alpha \gu^2+\gu^1-\gud^2-\gdu^2&=0
\end{align*}

\begin{align*}
\alpha \gud^1-\gud^2-\gd^1+\gu^1&=0 \\
\alpha\gud^2+\gud^1-\gd^2+\gu^2&=0
\end{align*}

\begin{align*}
\alpha \gdu^1-\gdu^2+\gu^1-\gd^1&=0 \\
\alpha \gdu^2+\gdu^1+\gu^2-\gd^2&=0
\end{align*}

\begin{align*}
\alpha \gd^1-\gd^2+\gud^1+\gdu^1&=0\\
\alpha \gd^2+\gd^1+\gud^2+\gdu^2&=0
\end{align*}
\end{multicols}

We can represent these systems in the following matrix array: \\
\begin{equation*}
\left[\begin{array}{rrrrrrrr}
\alpha &-1&-1&0&-1&0&0&0 \\
1 & \alpha & 0 & -1 & 0 & -1 & 0 &0\\
1&0&\alpha&-1&0&0&-1&0\\
0&1&1&\alpha&0&0&0&-1\\
1&0&0&0&\alpha&-1&-1&0 \\
0&1&0&0&1&\alpha&0&-1\\
0&0&1&0&1&0&\alpha&-1 \\
0&0&0&1&0&1&1&\alpha
\end{array}\right]
 \begin{bmatrix}
 \gu^1\\
 \gu^2\\
\gud^1\\
 \gud^2\\
 \gdu^1\\
 \gdu^2\\
 \gd^1\\
 \gd^2
 \end{bmatrix}=
 \begin{bmatrix}
 0\\0\\0\\0\\0\\0\\0\\0
 \end{bmatrix}
\end{equation*}
The determinant of the matrix is given by the polynomial $\alpha^8 + 12 \alpha^6 + 30\alpha^4 + 28 \alpha^2+9 $. Since this determinant is non-zero for $\alpha \geq 0$, then the system makes sense when $ \gij=0 $ for all $i, j, l \in \{1,2\}. $

\end{proof}

\subsection{Non-transitive Lie algebra actions on germs of surfaces}

\begin{table}[ht]
\begin{tabular}{|p{2.5cm}|p{9.3cm}|}
%\hline
%\multicolumn{2}{|c|}{ACTIONS WITH SPACE OF CONNECTIONS OF INFINITE DIMENSION} 
%\\
%\hline 
\hline 
Case & Christoffel symbols
\\
\hline \hline
9 & $\gij$ are functions of $y.$
\\
\hline
10& $\gij=0$ except $\gd^2, \gdu^1, \gud^1$  that are functions of y $y.$
\\
\hline
20, with $r=1$&
{
\begin{flushleft}
\begin{align*}
&\gd^1=\gdu^1=\gud^1=\gd^2=0,\\
&\gud^2+\gdu^2-\gu^1=-\frac{\partial_x^2 \xi_k}{\partial_x \xi_k},
\end{align*}
\end{flushleft}
}
with $\gu^2, \gud^2, \gdu^2$ y $\gu^1$ functions of $x.$
\\
\hline
21, with $r=1$&
{\begin{align*}
&\gd^1=\gdu^1=\gud^1=\gd^2=\gu^2=0,\\
&\gud^2+\gdu^2-\gu^1=-\frac{\partial_x^2 \xi_k}{\partial_x \xi_k},
\end{align*}
}
with $\gud^2, \gdu^2$ y $\gu^1$ functions of  $x.$
\\
\hline
\end{tabular}
\caption{Christofell symbols for invariant connections of non-transitive Lie algebra actions on surfaces in canonical coordinates.}\label{tabla:Christofell_nontransitive}
\end{table}

For non-transitive Lie algebra actions we have the following results.

\begin{proposition}\label{pr:no_transitive1}
Let us consider a non-transitive Lie algebra action of $\mathfrak g\simeq \r^{r+1}$ on an open subset $M\subseteq \mathbb R^2$ to case (20) in Table \ref{tabla:acciones no transitivas}. If ${\rm Cnx}(M)^{\mathfrak g}\neq \emptyset$ then the dimension of $\mathfrak g$ is equal to 2 (r=1). In such case, Christoffel symbols of $\mathfrak g$-invariant connections are characterized by the equations:
{\begin{align*}
&\gd^1=\gdu^1=\gud^1=\gd^2=0\\
&\gud^2+\gdu^2-\gu^1=-\frac{\partial_x^2 \xi}{\partial_x \xi},
\end{align*}
}
with $\gu^2, \gud^2, \gdu^2$ y $\gu^1$ functions of $x.$
\end{proposition}

\begin{proof}
Let us consider $\mathfrak g = \langle \partial_y, \xi_1(x)\partial_y,\ldots, \xi_r(x)\partial_y \rangle$. By taking components of $$\na{\xi_k(x)\dy}{\dxi}{\dxj} $$ 
we obtain the following system:
\begin{align*}
\left(\gud^2+\gdu^2-\gu^1\right)\dx \xi_k(x)+\dx^2\xi_k(x)&=0,
\\
\left(\gdu^1-\gd^2 \right)\dx \xi_k(x)&=0,
\\
\left(\gud^1+\gdu^1\right)\dx \xi_k(x)&=0,
\\
\left(\gd^2-\gud^1 \right)\dx \xi_k(x)&=0,
\\
\gd^1\dx \xi_k(x)&=0,
\\
\gd^2\dx \xi_k(x)&=0.
\end{align*}

This system is  compatible if for $k\neq l$ we have 
$$\frac{\partial_x^2 \xi_k}{\partial_x \xi_k} = 
\frac{\partial_x^2 \xi_l}{\partial_x \xi_l}.$$
So that functions $\xi_k(x)$ and $\xi_l(x)$ are related by an affine transformation $\xi_l(x) = a\xi_k(x) + b$. Therefore, by taking $\xi = \xi_1$ we have that $\mathfrak g$ is spanned by $\partial_y$, $\xi(x)\partial_y$ and it corresponds to the case $r=1$. We obtain that the $\Gamma^k_{ij}$ are functions of $x$ satisfying the following relations:
\begin{subequations}\label{21e1}
\begin{align}
&\gd^1=\gdu^1=\gud^1=\gd^2=0,\\
&\gud^2+\gdu^2-\gu^1=-\frac{\partial_x^2 \xi}{\partial_x \xi}.
\end{align}
\end{subequations}

\end{proof}

\begin{proposition}\label{pr:no_transitive2}
Let us consider a non-transitive Lie algebra action of $\mathfrak g \simeq \r\ltimes \r^{r+1}$ on an open subset $M\subseteq\mathbb R^2$ corresponding to case (21) in Table \ref{tabla:acciones no transitivas}. If ${\rm Cnx}(M)^{\mathfrak g}\neq \emptyset$ then the dimension of $\mathfrak g$ is equal to 3 (r=1). In such case, Christoffel symbols of $\mathfrak g$-invariant connections are characterized by the equations:
{\begin{align*}
&\gd^1=\gdu^1=\gud^1=\gd^2=\gu^2=0\\
&\gud^2+\gdu^2-\gu^1=-\frac{\partial_x^2 \xi}{\partial_x \xi}.
\end{align*}
}
with $\gud^2, \gdu^2$ y $\gu^1$ functions of  $x.$
\end{proposition}

\begin{proof}
This algebra contains that of case 20. Therefore, we have that $r=1$ and the system of equations \ref{21e1} are satisfied. There is an additional equation $\na{y\dy}{\dxi}{\dxj}$ yielding
$\gu^2=0.$ We conclude that $\gud^2, \gdu^2$ and $\gu^1$ are functions of $x$ and the following relations are satisfied: 
\begin{subequations}
\begin{align}
&\gd^1=\gdu^1=\gud^1=\gd^2=\gu^2=0,\\
&\gud^2+\gdu^2-\gu^1=-\frac{\partial_x^2 \xi_k}{\partial_x \xi_k}.
\end{align}
\end{subequations}
\end{proof}

Lie algebra actions corresponding to 9, 10, and 11 are well known one dimensional geometries. Case 11 corresponds to the projective geometry of the real line and it does not admit invariant connections. A direct computation yields the space of invariant connections in canonical coordinates.
Summarizing, the following Table \ref{tabla:Christofell_nontransitive}, contains the computation of Christoffel symbols of invariant connections in canonical coordinates.

\subsection{Transitive imprimitive Lie algebra actions}

\begin{theorem}\label{th:imprimitive}
Let us consider a faithful transitive imprimitive action of a Lie algebra $\mathfrak g$ in a connected surface $M$. Then, if the affine space ${\rm Cnx}(M)^{\mathfrak g}$ is not empty, then it falls in one of the following cases:
\begin{enumerate}
\item[(a)] The action corresponds to one of the following cases of Table \ref{tabla:imprimitive actions}:
    \begin{enumerate}
        \item[(i)] Case (12) with $\alpha\neq \frac{1}{2}$,
        \item[(ii)] Case (13),
        \item[(iii)] Case (24) with $r=1$,
        \item[(iv)] Case (25) with $r=1$,
        \item[(v)] Case (26) with $r=1$.
    \end{enumerate}
     For any such cases, the only invariant connection, in canonical coordinates, is the standard affine connection.
\item[(a')] The action corresponds to case (17). There is only an invariant connection whose Christoffel symbols, in canonical coordinates are:
\begin{align*}
\gu^2&=\gud^1=\gud^2=\gdu^1=\gdu^2=\gd^1=0,\\
\gu^1&=-\frac{2}{x-y},\\
\gd^2&=-\frac{2}{y-x}.    
\end{align*}
\item[(b)] The action corresponds to case (12) of Table \ref{tabla:imprimitive actions} with $\alpha =\frac{1}{2}$. The space ${\rm Cnx}(M)^{\mathfrak g}$ has dimension $1$. The equations for the Christoffel symbols in canonical coordinates are:
    $$\Gamma_{22}^1 \,\, \mathrm{cte}, \mbox{ for all other symbols }\Gamma_{ij}^k = 0.$$
\item[(c)] The action corresponds to case (18). The affine space ${\rm Cnx}(M)^{\mathfrak g}$ has dimension $3$. The equations for the Christoffel symbols in canonical coordinates are:
    {\begin{align*}
\gu^1&=\frac{a+b}{y^2},\quad \gu^2=\frac{c}{y^3},\\
\gud^2&=\frac{a}{y^2},\quad \gdu^2=\frac{b}{y^2},\\
\gd^2&=-\frac{2}{y}\quad\gd^1=0\\
\gud^1&=\gdu^1=-\frac{1}{y}.
\end{align*}}
    for arbitrary constants $a,b,c$.
\item[(d)] The action corresponds to case (22) with $r=1$ and $\eta_1(x) = e^{\alpha x}$. The affine space ${\rm Cnx}(M)^{\mathfrak g}$ has dimension $8$. The equations for the Christoffel symbols in canonical coordinates are:
    {\begin{align*}
\gu^2&=c_{22}^1\alpha^3y^3-c_{22}^2\alpha^2 y^2+(c_{11}^1-c_{21}^2-c_{12}^2)\alpha y-\alpha^2 y+c_{11}^2
\\
\gu^1&=c_{22}^1\alpha^2 y^2-(c_{12}^1+c_{21}^1)\alpha y+c_{11}^1
\\
\gud^2&=-c_{22}^1\alpha^2 y^2-(c_{22}^2-c_{12}^1)\alpha y+c_{12}^2
\\
\gdu^2&=-c_{22}^1\alpha^2 y^2-(c_{22}^2+c_{21}^2)\alpha y+c_{21}^2
\\
\gud^1&=-c_{22}^1\alpha y+c_{12}^1
\\
\gdu^1&=-c_{22}^1\alpha y+c_{21}^1
\\
\gd^2&=c_{22}^1\alpha y+c_{22}^2
\\
\gd^1&=c_{22}^1.
\end{align*}} 
    for arbitrary constants $c_{ij}^k$, with $i,j,k \in \{1,2\}$.
\item[(e)] The action corresponds to case (23) with $r=1$ and $\eta_1(x) = e^{\alpha x}$.  The affine space ${\rm Cnx}(M)^{\mathfrak g}$ has dimension $4$. The equations for the Christoffel symbols in canonical coordinates are:
{\begin{align*}
\gd^2&=\gud^1=\gdu^1=\gd^1=0\\
\gu^1&=a, \gud^2=c, \gdu^2=d  \\
\gu^2&=-\alpha (d+c-a)y-\alpha^2y+b.
\end{align*}
}
for arbitrary constants $a,b,c,d.$
\end{enumerate}
\end{theorem}

\begin{proof}
The proof is based in an analysis of the Lie derivative of a general connections by the generators of the Lie algebra action in canonical coordinates for each case appearing in Table \ref{tabla:imprimitive actions}.

%{\color{red}{desde acá}}
\begin{enumerate}
\item [(a)] Let  $\mathfrak g_{6}$, $\mathfrak g_{12,\alpha\neq 1/2}$, $\mathfrak g_{13}$, $\mathfrak g_{24,r=1}$, $\mathfrak g_{25,r=1}$, $\mathfrak g_{26,r=1}$ be Lie algebras correspond to their corresponding cases in the tables  \ref{tabla:acciones primitivas}, \ref{tabla:imprimitive actions}. By the last result, the only invariant connection for the algebra $\mathfrak g_{6}$ is the affine standard connection. As the other algebras are contained in this algebra forming the lattice:
$$
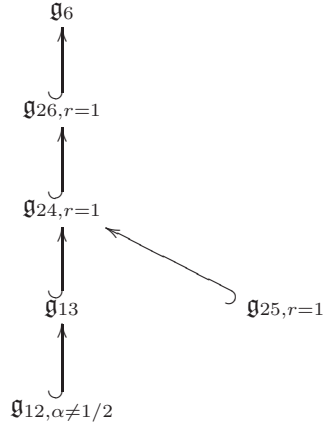
\xymatrix{
\mathfrak{g}_6 \\ \mathfrak g_{26,r=1} \ar@{^{(}->}[u] \\ \mathfrak g_{24,r=1} \ar@{^{(}->}[u] \\ \mathfrak g_{13} \ar@{^{(}->}[u] & \ \ \ \ \ \ \ \  \mathfrak g_{25,r=1} \ar@{^{(}->}[ul] \\  \mathfrak g_{12,\alpha\neq 1/2} \ar@{^{(}->}[u]
}$$
then is sufficient to show the result for the cases
 $12$ with  $\alpha\neq 1/2$, $25$ with $r=1$.
 
\begin{itemize}

\item Case (12), $\g=\left \{\dx, \dy,x\dx+\alpha y\dy\right \},$ $0<|\alpha|\leq 1.$

The coefficients $\gij$ are constants by \ref{constante}. If $X=x\dx+\alpha y\dy$, with $\alpha \neq \frac{1}{2}$, then from that $\na{X}{\dxi}{\dxj}$ it follows  $\gij=0$.

\item Case (25), $r=1$. in this case the algebra is $\g=\{\dx, \dy, x\dx+\left(y+x\right)\dy \}$. As $\dx$ and $\dy$ are in $ \g$, then $\gij$ are constants. Taking the field $Y=\dx+\left(y+x\right)\dy$, from $\na{Y}{\dxi}{\dxj}$, we have the system
\begin{align*}
-\gu^1&+\gu^2+\gud^2+\gdu^2=0,\\
\gu^1&+\gud^1+\gdu^1=0,\\
\gdu^2&-\gdu^1+\gd^2=0,\\
\gud^2&-\gud^1+\gd^2=0,\\
\gud^1&+\gd^1=0,\\
\gdu^1&+\gd^1=0,\\
\gd^2&=\gd^1=0.
\end{align*}
With solution $\gij=0$.

\end{itemize}

\item[(a')]  Case (17), $\g=\left\{\dx+\dy,  x\dx+y\dy, x^2\dx+y^2\dy \right\}$. 
Considering $X=\dx+\dy$, $Y=x\dx+y\dy$ and  $Z=x^2\dx+y^2\dy$, 
from $\na{X}{\dxi}{\dxj}$ and $\na{Y}{\dxi}{\dxj}$, we obtain for $i,j,l\in \{1,2\}$  the system
\begin{align*}
\dx\gij=-\dy\gij,\\
(x-y)\dx\gij+\gij=0,\\
(y-x)\dy\gij+\gij=0.
\end{align*}
Replacing those equations in  $\na{Z}{\dxi}{\dxj}$, we get the system
\begin{align*}
\gu^2&=\gud^1=\gud^2=\gdu^1=\gdu^2=\gd^1=0,\\
\gu^1&=-\frac{2}{x-y},\\
\gd^2&=-\frac{2}{y-x}.    
\end{align*}
%N\'otese que la recta $x=y$ debe interpretarse como la frontera del espacio infinitesimalmente homog\'eneo en cuesti\'on, y por este motivo es posible encontrar una singularidad en la conexi\'on.

\item[(b)] We have that $\gij$ are constants, from $\na{X}{\dxi}{\dxj}$, with $X=x\dx+\alpha y\dy$, we have the result.

\item[(c)] Case (12), $\g=\{\dx, 2x\dx+y\dy, x^2\dx+xy\dy\}$. As $\dx$ is in the algebra, then all the symbols  $\gij$ are functions of  $y$. Taking  $Y=2x\dx+y\dy$ and $Z=x^2\dx+xy\dy$,
from $\na{Y}{\dxi}{\dxj}$ and  $\na{Z}{\dxi}{\dxj}$, we have the system:

\begin{multicols}{2}
\begin{align*}
y\dy\gu^1+2\gu^1&=0,\\
y\dy\gu^2+3\gu^2&=0,\\
y\dy\gud^1+\gud^1&=0,\\
y\dy\gud^2+2\gud^2&=0, \\
y\dy\gdu^1+\gdu^1&=0,\\
y\dy\gdu^2+2\gdu^2&=0%,\\
%y\gdu^1+y\gud^1+2&=0.
\end{align*}

\begin{align*}
\gud^2+\gdu^2-\gu^1&=0,\\
y\gd^2-y\gud^1+1&=0,\\
y\gd^2-y\gdu^1+1&=0, \\
y\gdu^1+y\gud^1+2&=0,\\
\gd^1&=0.
\end{align*}
\end{multicols}

Where the general solution are expressed in function of  arbitrary constants 
 $a$, $b$, $c$ $\in\r$:   

\begin{align*}
\gu^1&=\frac{a+b}{y^2},\quad \gu^2=\frac{c}{y^3},\\
\gud^2&=\frac{a}{y^2},\quad \gdu^2=\frac{b}{y^2},\\
\gd^2&=-\frac{2}{y}\quad\gd^1=0,\\
\gud^1&=\gdu^1=-\frac{1}{y}.
\end{align*}

\item[(d)] $\g=\left\langle\dx, \eta_1(x)\dy, \cdots, \eta_r(x)\dy \right\rangle,$ with $r\geq 1.$

Taking  $k$ such that  $1\leq k \leq r$.
The symbols  $\gij$ depends of $y$ because $\dx$ is in the algebra. From  $\na{\eta_k(x)\dy}{\dxi}{\dxj}$, obtain the system 

\begin{align*}
\eta_k(x)\dy\gu^1+\dx\eta_k(x)\left(\gud^2+\gdu^2-\gu^1\right)+\dx^2\eta_k(x)&=0,
\\
\eta_k(x)\dy\gu^1+\dx\eta_k(x)\left(\gdu^1+\gud^1\right)&=0,
\\
\eta_k(x)\dy\gud^1+\gd^1\dx\eta_k(x)&=0,
\\
\eta_k(x)\dy\gud^2+\dx\eta_k(x)\left(\gd^2-\gud^1\right)&=0,
\\
\eta_k(x)\dy\gdu^1+\gd^1\dx\eta_k(x)&=0,
\\
\eta_k(x)\dy\gdu^2+\dx\eta_k(x)\left(\gd^2-\gdu^1\right)&=0,
\\
\eta_k(x)\dy\gd^2-\gd^1\dx\eta_k(x)&=0,
\\
\eta_k(x)\dy\gd^1&=0.
\end{align*}

Using this equations we obtain that  $\alpha_k:=\frac{\dx\eta_k(x)}{\eta_k(x)}$ and $\alpha^2_k:=\frac{\dx^2\eta_k(x)}{\eta_k(x)}$ are constants.
Therefore $\eta_k(x)$ is  multiple of  $e^{\alpha_kx}$.
Finally, for  $k$ fixed, we can solve the symbols in function of  $8$ arbitrary constants  $c_{ij}^k$:

\begin{subequations}
\begin{align}
\gu^2&=c_{22}^1\alpha^3_ky^3+(c_{22}^2-c_{12}^1-c_{21}^1)\alpha^2_ky^2+(c_{11}^1-c_{21}^2-c_{12}^2)\alpha_ky-\alpha^2_ky+c_{11}^2, \label{simbolos22}
\\
\gu^1&=c_{22}^1\alpha^2_ky^2-(c_{12}^1+c_{21}^1)\alpha_ky+c_{11}^1,
\\
\gud^2&=-c_{22}^1\alpha^2_ky^2-(c_{22}^2-c_{12}^1)\alpha_ky+c_{12}^2,
\\
\gdu^2&=-c_{22}^1\alpha^2_ky^2-(c_{22}^2-c_{21}^1)\alpha_ky+c_{21}^2,
\\
\gud^1&=-c_{22}^1\alpha_ky+c_{12}^1,
\\
\gdu^1&=-c_{22}^1\alpha_ky+c_{21}^1,
\\
\gd^2&=c_{22}^1\alpha_ky+c_{22}^2,
\\
\gd^1&=c_{22}^1. \label{simbolos22b}
\end{align}
\end{subequations}

This system is compatible if the constants $\alpha_k$ are equals for all $k$. In this case the functions  $\eta_k(x)$ spam a space of dimension $1$, so $r=1$. That is, if $r=1$ and $\eta_1(x) = e^{\alpha_k x}$, there is a  $8$-dimensional space of invariant connections. Otherwise, there is not invariant connections.

\item[(e)] 
As this algebra contains the algebra of the case (22), then  $r=1$ and $\eta_1(X) = e^{\alpha_1 x}$. We obtain a system of equations including 
\eqref{simbolos22} -- \eqref{simbolos22b} and also from $\na{y\dy}{\dxi}{\dxj}$, we obtain additional equations: 
\begin{align}\label{e23-1}
y\dy\gu^1=y\dy\gud^2=y\dy\gdu^2=0
\end{align}
and the system  
\begin{subequations}\label{e23-2}
\begin{align}
y\dy\gu^1-\gu^1&=0\label{e23-2-1},
\\
y\dy\gud^1+\gud^1&=0\label{e23-2-2},
\\
y\dy\gdu^1+\gdu^1&=0\label{e23-2-3},
\\
y\dy\gd^1+2\gd^1&=0\label{e23-2-4},
\\
y\dy\gd^2+\gd^2&=0\label{e23-2-5}.
\end{align}
\end{subequations}
From \ref{e23-1}, $\gu^1, \gud^2, \gdu^2$ are constants. As $\gd^1$ is constant, from \ref{e23-2-4}, $\gd^1=0$. Combining this with the equations from the previous case we obtain $\gd^2=\gdu^1=\gud^1=0$. So we have: 
\begin{align*}
\gd^2&=\gud^1=\gdu^1=\gd^1=0,\\
\gdu^2&=c_{21}^2, \gud^2=c_{12}^2, \gu^1=c_{11}^1,\\
\gu^2&=-\alpha_1(c_{21}^2+c_{12}^2-c_{11}^1)y-\alpha ^2_1y+c_{11}^2,
\end{align*}
where  $c_{21}^2, c_{12}^2, c_{11}^1$ and $c_{11}^2$ are arbitrary constants.

\end{enumerate}

It remains to check that all other cases of transitive imprimitive infinitesimal actions, namely cases (14), (15), (16), (19), (27) and (28) do not admit invariant connections.
The Lie algebra of case (14) is contained in that of cases (15) and (16). Thus, it suffices to show the case (14). In such case, since $\dx$ and $\dy$ are in the algebra, then $\gij$ are constants. From $\na{x\dx}{\dxi}{\dxj}$, $\gij=0$ , except the  symbols $\gud^1, \gdu^1, \gd^2$. As $\na{x^2\dx}{\dx}{\dx}$ implies that $2\dx=0$, this is impossible, then in this case there is not connection.

For the case (19),  $\gij$ depends of $y$. From $\na{x\dx}{\dxi}{\dxj}$ it follows that  $\gij=0$, except  $\gud^1, \gdu^1, \gd^2$. 
From $\na{y\dy}{\dxi}{\dxj}$, we have the system 
\begin{align*}
y\dy\gud^1+\gud^1&=0\\
y\dy\gdu^1+\gdu^1&=0\\
y\dy\gd^2+\gd^2&=0
\end{align*}
Using this in $\na{Y}{\dxi}{\dxj}$, with $Y=x^2\dx+xy\dy$ we have the system 
\begin{align*}
-y\gud^1+y\gd^2+1&=0\\
-y\gdu^1+y\gd^2+1&=0\\
xy\dy\gd^2+y\gd^2+1&=0\\
y\gdu^1+y\gud^1+2&=0\\
\gdu^1&=0
\end{align*}
Which is inconsistent, in this case there is not connection.

In the case (27), $\gij$ are constants. For $Y=2x\dx+ky\dy$ with $1\leq k\leq r$, from $\na{Y}{\dxi}{\dxj}$ we have the system:
\begin{align*}
2\gu^1\dx+(4-k)\gu^2\dy=0,\\
k\gud^1\dx+2\gud^2\dy=0,\\
k\gdu^1\dx+2\gdu^2\dy=0,\\
(2k-2)\gd^1\dx+k\gd^2\dy=0.
\end{align*}
If $k=1$, then $\gij=0$ except $\gd^1$,  but from $\na{x\dy}{\dy}{\dy}$, $\gd^1=0.$
If $r=2$, from $\na{x^2\dy}{\dx}{\dx}$ we have a contradiction. So $r=1$, as $\gij=0$ from $\na{Z}{\dx}{\dx}$, with $Z=x^2\dx+xy^2\dy$ we have also a contradiction. Therefore there are not invariant connections in this case.

Finally the Lie algebra of case (28) is contained in that of case (13), then  $\gij=0$. From $\na{Z}{\dx}{\dx}$, with $Z=x^2\dx+rxy\dy$, we obtain the incompatibility of the system for the Christoffel symbols and there are no invariant connections.
\end{proof}

\section{Homogeneous surfaces}

By an homogeneous manifold we understand an smooth manifold $M$ endowed with a transitive smooth action $\rho\colon G\to {\rm Diff}(M)$ of a connected Lie group $G$. As it is well known the surface can be recovered as the quotient $M\simeq G/H$ where $H$ is the stabilizer of a point in $M$. We say that two action 
$\rho\colon G\to {\rm Diff}(M)$ and $\rho'\to {\rm Diff}(M')$ are equivalent if they induce the same transformation group in $M$, that is, $\rho(G) = \rho(G')$. Any action is always equivalent to a faithful action: we may consider the exact sequence,
$${\rm Id}\to {\rm ker}(\rho) \to  G \xrightarrow{\rho} {\rm Diff}(M)$$
and replace $G$ by $\bar G = G/\ker{\rho}$.
Given a faithful action of $G$ on $M$ we can replace $G$ by its universal cover. The action of the universal cover of $G$ in $M$ is not faithful but at least infinitesimally faithful. Therefore, there is no loss of generality in assuming that the $G$ is simply connected and the action of $G$ on $M$ is infinitesimally faithful. 

\begin{lemma}
Let $H$ be a Lie subgroup of a simply connected Lie group $G$. Then, the homogeneous manifold $G/H$ is simply connected if and only if $H$ is connected. 
\end{lemma}

\begin{proof}
First let us see that if $G/H$ is simply connected, then $H$ is connected. Reasoning by contradiction let us assume that $H$ is not connected a, and let $H_0$ be its connected component of the the identity. Then, $G/H_0\to G/H$ is a non-trivial connected (since $G$ is connected) covering space. Therefore $G/H$ is not simply connected.

Let us assume now that $H$ is connected. Let us denote by $x_0$ the class $H$ in $G/H$. Let $\gamma\colon [0,1]\to G/H$ is a loop based on $x_0$. We can lift $\gamma$ to a path $\tilde \gamma$ in $G$ such that $\gamma(0) = e$ and $\gamma(1)\in H$. Since $H$ is connected there is path $\tau\colon [0,1]\to H$ such that $\tau(0) = e$ and $\tau(1) = \tilde\gamma(1)$. Since $G$ is simply connected there is an homotopy with fixed extremal points $\tilde\gamma \sim \tau$. The projection of the homotopy onto $G/H$ tell us that $\gamma$ is contractible. 
\end{proof}

Therefore, simply connected homogeneous manifolds arise as quotients of simply connected Lie groups by connected Lie subgroups. On the other hand let us consider $M$ an homogeneous manifold, and $\pi\colon \tilde M \to M$ its universal cover. Since each diffeomorphism of $M$ can be lifted, up to choice of two points in the fibers of $\pi$, to a diffeomorphism of $\tilde M$, we have an exact sequence:
$${\rm Id} \to {\rm Aut}(\tilde M/M)\to {\rm Diff}(\tilde M/M)\xrightarrow{\pi_*} {\rm Diff}(M) \to {\rm Id}$$
where ${\rm Diff}(\tilde M/M)$ is the group of diffeomorphisms of $\tilde M$ that respect the fibers of $\pi$. Therefore, by taking $\tilde G$ the connected component of $\pi_*^{-1}(G)$, we have that $\tilde M$ is a simply connected homogeneous manifold with the action of $\tilde G$. Therefore, homogeneous manifold can be seen as quotients of simply connected homogeneous manifolds.

The analysis of the invariant connections for Lie algebra actions allow us to classify, up to equivalence, all the simply connected homogeneous surfaces having more than one, or exactly one, invariant connections. 

\begin{theorem}\label{th:hs}
Let $M$, endowed with an action of a connected Lie group $G$, be a simply connected homogeneous surface. Let us assume that $M$ admits at least two $G$-invariant connections. Then, $M$ is equivalent to one of the following cases:
\begin{enumerate}
\item[(a)] The affine plane $\mathbb R^2$ with one of the following transitive groups of affine transformations:
$${\rm Res}_{(2:1)}(\mathbb R^2) = \left\{ \left[ \begin{array}{c}x\\ y\end{array}\right] \mapsto A \left[ \begin{array}{c}x\\ y\end{array}\right] +  \left[ \begin{array}{c}b_1\\ b_2\end{array}\right]\mid A = 
\left[ \begin{array}{cc} \lambda^2 & 0 \\ 0 & \lambda \end{array} \right],\,\lambda>0
\right\},$$
$${\rm Trans}(\mathbb R^2) = \left\{  \left[ \begin{array}{c}x\\ y\end{array}\right] \mapsto  \left[ \begin{array}{c}x\\ y\end{array}\right] +  \left[ \begin{array}{c}b_1\\ b_2\end{array}\right] \mid b_1,b_2\in\mathbb R \right\},$$
$${\rm Res}_{(0:1)}(\mathbb R^2) = \left\{  \left[ \begin{array}{c}x\\ y\end{array}\right] \mapsto A  \left[ \begin{array}{c}x\\ y\end{array}\right] +  \left[ \begin{array}{c}b_1\\ b_2\end{array}\right]\mid A = 
\left[ \begin{array}{cc} 1 & 0 \\ 0 & \lambda \end{array} \right],\,\lambda>0
\right\}.$$
%$G/H$ where $G = \mathbb R\ltimes \mathbb R^2$ with the semi-direct product %law 
%$$(x,y,z)\star(x',y',z') = (x+x',y + e^xy', z + e^{x/2}z'),$$ 
%and $H =  \mathbb R  \ltimes 0.$ 
\item[(b)] ${\rm SL}_2(\mathbb R)/U$ where $U$ is the subgroup of superior unipotent matrices
$$ U = \left\{ A\in {\rm SL}_2(\mathbb R) \mid A = 
\left[ \begin{array}{cc} 1 & \lambda \\ 0 & 1 \end{array} \right]
\right\}.$$
\item[(c)] $\mathbb R\ltimes \mathbb R$ acting on itself by left translations. 
\item[(d)] $G/H$ with $G = \mathbb R^2\ltimes \mathbb R$ and $H = (\mathbb R\times 0)\ltimes 0$. 
\end{enumerate}
%In such cases we also have that the dimension of the space of $G$-invariant connections is $1$, $3$, $8$ and $4$ respectively. 
\end{theorem}

\begin{proof}
Let $M$, endowed with an action of a connected Lie group $G$, be a simply connected homogeneous surface. Let us consider the infinitesimal lie algebra action induced by the action of $G$ in $M$. Since $G$ is connected then the space of $G$-invariant connections coincide with that of ${\rm Lie}(G)$-invariant connections. 

In virtue of Theorems \ref{th:primitive} and  \ref{th:imprimitive} if the infinitesimal action has more than one invariant connection, then it corresponds to cases (12) with $\alpha = \frac{1}{2}$, (18), (22) with $r=1$ or (23) with $r=1$ in table \ref{tabla:imprimitive actions}. Without loss of generality, we assume that $G$ is simply connected. Then, by third Lie theorem, it is completely determined by its Lie algebra. Therefore, by integrating the respective Lie algebras, we deduce the following:
\begin{itemize}
    \item Any homogeneous surface corresponding to case (12) with $\alpha\neq 0$ or cases (22, 23) with trivial semidirect product in Table \ref{tabla:imprimitive actions}, is equivalent to one listed in the case (a) of the statement;
    \item any homogeneous surface corresponding to case (18) in Table \ref{tabla:imprimitive actions}, is equivalent to that of case (b);
    \item any homogeneous surface corresponding to case (22) with $r=1$ in Table \ref{tabla:imprimitive actions} and non trivial semidirect product is equivalent to that of case (c);
    \item any homogeneous surface corresponding to case (23) with $r=1$ in Table \ref{tabla:imprimitive actions} and non trivial semidirect product is equivalent of that of case (d).
\end{itemize}
This completes the proof.
\end{proof}

\begin{remark}
In the cases (c) and (d) of Theorem \ref{th:hs} we assume that the semidirect product is not trivial. Otherwise we fall in the case (a).
\end{remark}

\begin{theorem}\label{th:hs2}
Let $M$, with the action of $G$, be a simply connected homogeneous surface. Let us assume that $M$ admits exactly one $G$-invariant connection. Then, $M$ is equivalent to one of the following cases:
\begin{enumerate}
\item[(a)] The affine plane $\mathbb R^2$ with $G$ any connected transitive subgroup of the group ${\rm Aff}(\mathbb R^2)$ of affine transformations containing the group of translations and not conjugated with any of the groups,
$${\rm Res}_{(2:1)}(\mathbb R^2) = \left\{ \left[ \begin{array}{c}x\\ y\end{array}\right] \mapsto A \left[ \begin{array}{c}x\\ y\end{array}\right] +  \left[ \begin{array}{c}b_1\\ b_2\end{array}\right]\mid A = 
\left[ \begin{array}{cc} \lambda^2 & 0 \\ 0 & \lambda \end{array} \right],\,\lambda>0
\right\},$$
$${\rm Trans}(\mathbb R^2) = \left\{  \left[ \begin{array}{c}x\\ y\end{array}\right] \mapsto  \left[ \begin{array}{c}x\\ y\end{array}\right] +  \left[ \begin{array}{c}b_1\\ b_2\end{array}\right] \mid b_1,b_2\in\mathbb R \right\},$$
$${\rm Res}_{(0:1)}(\mathbb R^2) = \left\{  \left[ \begin{array}{c}x\\ y\end{array}\right] \mapsto A  \left[ \begin{array}{c}x\\ y\end{array}\right] +  \left[ \begin{array}{c}b_1\\ b_2\end{array}\right]\mid A = 
\left[ \begin{array}{cc} 1 & 0 \\ 0 & \lambda \end{array} \right],\, \lambda>0
\right\}.$$
(this case includes the euclidean plane)

\item[(b)] ${\rm SL}_2(\mathbb R)/H$ where $H$ is the group of special diagonal matrices,
$$H = \left\{ A \in {\rm SL}_{2}(\mathbb R) \mid  
A =  \left[ \begin{array}{cc} \lambda & 0 \\ 0 & \lambda^{-1} \end{array} \right], \,\, \lambda>0 \right\}.$$
\item[(c)] The hyperbolic plane 
$$\mathbb H = \{z\in\mathbb C \mid {\rm Im}(z)>0 \}$$ 
with the group ${\rm SL}_2(\mathbb R)$ of hyperbolic rotations.
\item[(d)] Spherical surface, 
$$S^2 = \{(x,y,z)\in \mathbb R^3\mid x^2 + y^2 + z^2 =1 \}$$ 
with its group ${\rm SO}_3(\mathbb R)$ of rotations.
\end{enumerate}
\end{theorem}

\begin{proof}
We follow the same reasoning that in the proof of Theorem \ref{th:hs}. 
In virtue of Theorems \ref{th:primitive} and  \ref{th:imprimitive} if the infinitesimal action has exactly one one invariant connection, then it corresponds to cases (1), (3), (5), (6), (12) with $\alpha \neq \frac{1}{2}$, (13), (24) with $r=1$, (25) with $r=1$, (26) with $r=1$, (17), (2) or (3). Then we check that cases (1), (3), (6), (5),(12) with $\alpha\neq 1/2$, (13), (24) with $r=1$, (25) with $r=1$ and (26) with $r=1$, all ot them correspond to case (a) in the statement. Finally, case (17) corresponds to case (b), case (2) corresponds to case (c) and case (3) corresponds to case (d).
\end{proof}

\begin{remark}
Note that the hyperbolic plane, case (c) in Theorem \ref{th:hs2}, corresponds to the remaining $2$-dimensional simply connected quotient of ${\rm SL}_2(\mathbb R)$; $\mathbb H \simeq {\rm SL}_2(\mathbb R)/H$ where
$$H = \left\{  \left[ \begin{array}{cc} a & -b \\ b & a \end{array} \right] \mid a^2 + b^2 = 1 \right\}.$$
\end{remark}

\subsection*{Acknowledgements}

The authors acknowledge the support of their host institutions Universidad de Antioquia and Universidad Nacional de Colombia - Sede Medell\'in. The research of D.B.-S. has been partially funded by Colciencias project ''Estructuras lineales en geometr\'ia and Topolog\'ia'' 776-2017 57708 (HERMES 38300).

\bibliographystyle{plain}
\bibliography{references}
\end{document}